\documentclass[11pt]{amsart}

\usepackage{amsmath,amssymb,amsthm}
\usepackage{amsmath,amssymb,amsthm,amscd}
\usepackage[frame,cmtip,arrow,matrix,line,graph,curve]{xy}
\usepackage{graphpap, color}
\usepackage[mathscr]{eucal}
\usepackage{color}
\usepackage{soul}

\def\dfrac{\displaystyle\frac}
\def\dsum{\displaystyle\sum}

\newtheorem{prop}{Proposition}
\newtheorem{theo}[prop]{Theorem}
\newtheorem{lemm}[prop]{Lemma}

\newtheorem{defi}[prop]{Definition}
\newtheorem{conj}[prop]{Conjecture}

\newcommand{\al}{\alpha}
\newcommand{\abc}[1]{\left( #1 \right)}%
\renewcommand{\leq}{\leqslant}
\renewcommand{\geq}{\geqslant}

\newcommand{\p}{\partial}

\numberwithin{equation}{section}

\begin{document}

\title{Notes on the curvature estimates for Hessian equations}

\author{Changyu Ren}
\address{School of Mathematical Science\\
Jilin University\\ Changchun\\ China}
\email{rency@jlu.edu.cn}
\author{Zhizhang Wang}
\address{School of Mathematical Science\\ Fudan University \\ Shanghai, China}
\email{zzwang@fudan.edu.cn}
\thanks{Research of the first author is supported by NSFC Grant No. 11871243 and the second author is supported  by NSFC Grants No.11871161 and 11771103.}
\begin{abstract}
The main result of this paper gives a plenary proof on the curvature estimates for $k$ curvature equations with general
right hand sides with $n<2k$ based on a concavity inequality. We further give a explicit lower bound of the inequality.  
\end{abstract}
\maketitle

\section{introduction}

In this paper, we continue to  study the longstanding problem about the global curvature estimates for curvature equations with general right hand side
\begin{equation}\label{1.1}
\sigma_k(\kappa(X))=\psi(X,\nu(X)), \quad \forall  X\in M,
\end{equation}
where $\sigma_k$ is the $k$-th elementary symmetric function, $\nu(X)$ and $\kappa(X)$ denote the outer normal vector and the principal curvatures of the hypersurface $X: M\rightarrow \mathbb R^{n+1}$, respectively. This problem was clearly posed by Guan-Li-Li in \cite{GLL} at first. Moreover, it is  very nature  to  consider the equation \eqref{1.1} with the right hand side containing the normal vector or in other words, gradient terms.

Equation \eqref{1.1} is associated with many important geometric problems.
In particular, the  famous  Minkowski problem, namely, the prescribed Gauss-Kronecker curvature on the outer normal, has been widely discussed by Nirenberg \cite{N}, Pogorelov \cite{P3}, Cheng-Yau \cite{CY}.
 Alexandrov \cite{A2, gg} also posed the problem of prescribing general Weingarten curvature on the outer normal.
Moreover, the prescribing curvature measure problem in convex geometry
has been extensively studied by Alexandrov \cite{A1}, Pogorelov
\cite{P1}, Guan-Lin-Ma \cite{GLM}, Guan-Li-Li \cite{GLL}, while the
prescribing mean curvature problem and Weingarten curvature problem
also have been considered and obtained fruitful results by
Bakelman-Kantor \cite{BK}, Treibergs-Wei \cite{TW}, Oliker \cite{O},
Caffarelli-Nirenberg-Spruck \cite{CNS4, CNS5}. More geometric applications can  be found in
\cite{LO,BL,AB,PPZ1,PPZ2,GL,X,BIS}, etc. Very recently, Ren-Wang-Xiao \cite{RWXiao} obtained the convexity of bounded entire space like hypersurfaces with constant $\sigma_{n-1}$ curvature in Minkowski space and constructed a lot of examples of these type, by using some techniques developed in \cite{GRW, LiRW} and \cite{RW}.

The $C^2$ a prior estimate for \eqref{1.1} has been studied extensively.
 When $\psi$ is independent of the normal vector, the $C^2$-estimate
was obtained by Caffarelli-Nirenberg-Spruck \cite{CNS3} for a general class of fully nonlinear operators.
Ivochkina \cite{I1,I} considered the Dirichlet problem of equation (1.1) on domains in $\mathbb{R}^n$, the $C^2$ estimate was proved there under some extra conditions on the dependence of $\psi$ on $\nu$.
The Pogorelov type interior $C^2$ estimate for the Hessian equation have been obtained by Chou-Wang \cite{CW}.
Sheng-Urbas-Wang \cite{SUW} obtained the Pogorelov type interior $C^2$ estimate for the curvature equation  of the graphic hypersurface.
$C^2$ estimates for the complex Hessian equations defined on K\"ahler manifolds have been obtained by Hou-Ma-Wu  \cite{HMW}.
The $C^2$ estimate was also established for the equation of the prescribed curvature measure problem by Guan-Li-Li \cite{GLL} and Guan-Lin-Ma \cite{GLM}.
If the function $\psi$ is  convex with respect to the normal, the global $C^2$ estimate is well known, which is obtained by Guan \cite{B}. Recently, Guan \cite{B2}  obtained an important result on  $C^2$ estimates for some  fully nonlinear equations defined on Riemannian manifolds.

In recent years, the authors have made many progresses on establishing $C^2$ estimates for equation \eqref{1.1}. More precisely,
Guan-Ren-Wang \cite{GRW} obtained the global curvature estimate of the closed convex hypersurface and the star-shaped $2$-convex hypersurface.
The corresponding case in complex setting has been established by Phong-Picard-Zhang \cite{PPZ3} on the K\"ahler manifold and Dong \cite{D} on the Hermitian mainifold. Li-Ren-Wang \cite{LRW} improved the convex condition  to $k+1$- convex condition for any Hessian equations and derived the Pogorelov type interior $C^2$ estimates. For the case $k=n-1$, Ren-Wang \cite{RW} obtained the global curvature estimates of $n-1$ convex solutions for $n-1$ Hessian equations and completely solved the longstanding problem.
Chen-Li-Wang \cite{CLW} established the global curvature estimate for the prescribed curvature problem in arbitrary warped product spaces. Li-Ren-Wang \cite{LiRW} considered the global curvature estimate of convex solutions for a class of general  Hessian equations.
Spruck-Xiao \cite{SX} obtained the
curvature estimate for the prescribed scalar curvature problem in space forms and
gave a simple proof of Theorem 1.6 in \cite{GRW}.

Before starting our main theorem, we need to introduce the admissible set for equation \eqref{1.1}. Following \cite{CNS3}, we define an open,
convex, symmetric (invariant under the interchange of any two
$\kappa_i$) cone with vertex at the origin,  containing the positive
cone, $\Gamma_{+}=\{\kappa\in \mathbb{R}^n ;$ each component $\kappa_i>0,1\leq
i\leq n\}$:
\begin{defi}\label{k-convex} For a domain $\Omega\subset \mathbb R^n$, a function $v\in C^2(\Omega)$ is called $k$-convex if the eigenvalues $\kappa (x)=(\kappa_1(x), \cdots, \kappa_n(x))$ of the hessian $\nabla^2 v(x)$ is in $\Gamma_k$ for all $x\in \Omega$, where $\Gamma_k$ is the G\r{a}rding's cone
\begin{equation}\label{def G2}
\Gamma_k=\{\kappa \in \mathbb R^n \ | \quad \sigma_m(\kappa)>0,
\quad  m=1,\cdots,k\}.\nonumber
\end{equation}
A $C^2$ regular hypersurface $M\subset R^{n+1}$ is called $k$-convex if its principal curvature vector
$\kappa(X)\in \Gamma_k$ for all $X\in M$.
\end{defi}

The purpose of the present paper is to establish the global curvature estimate based on the following concavity conjecture:

\begin{conj} \label{con}
Assume that $\kappa=(\kappa_1,\cdots,\kappa_n)\in\Gamma_{k}$ with
$n<2k$, $\kappa_1$ is the maximum entry of $\kappa$, and
$\sigma_{k}(\kappa)$ has the absolutely positive lower bound and
upper bound, $N_0\leq \sigma_{k}(\kappa)\leq N_1$. For any given
index $1\leq i\leq n$, if $\kappa_i>\kappa_1-\sqrt{\kappa_1}/n$, the
following quadratic form is non negative,
\begin{align}\label{s3.01}
&\kappa_i\Big[K\Big(\sum_j\sigma_{k}^{jj}(\kappa)\xi_j\Big)^2-\sigma_{k}^{pp,qq}(\kappa)\xi_{p}\xi_{q}\Big]
-\sigma^{ii}_{k}(\kappa)\xi_{i}^2+\sum_{j\neq i}a_j\xi_{j}^2\geq 0,
\end{align}
for any $n$ dimensional vector
$\xi=(\xi_1,\xi_2,\cdots,\xi_n)\in\mathbb{R}^n$, when $\kappa_1$ and
the constant $K$ are sufficiently large. Here $a_j$ is defined by
\begin{eqnarray}\label{aj}
a_j=\sigma_{k}^{jj}(\kappa)+(\kappa_i+\kappa_j)\sigma_{k}^{ii,jj}(\kappa).
\end{eqnarray}
\end{conj}
Here, the notations $\sigma_k^{ii}(\kappa),\sigma_k^{jj}(\kappa), \sigma_k^{pp,qq}(\kappa),\sigma_k^{ii,jj}(\kappa)$ mean
$$\sigma_k^{ii}(\kappa)=\frac{\p \sigma_k(\kappa)}{\p \kappa_i},\sigma_k^{jj}(\kappa)=\frac{\p \sigma_k(\kappa)}{\p \kappa_j}, \sigma_k^{pp,qq}(\kappa)=\frac{\p^2 \sigma_k(\kappa)}{\p \kappa_p\p \kappa_q},\sigma_k^{ii,jj}(\kappa)=\frac{\p^2 \sigma_k(\kappa)}{\p \kappa_i\p \kappa_j}.$$

Note that, in \cite{RW} and \cite{RWarxiv}, we have proved the Conjecture for $k=n-1$ and $k=n-2$. If $k=n$, the inequality \eqref{s3.01} is well known. Thus, until to now, the above Conjecture holds when $k\geq n-2$.

The main theorem of this paper is following:
\begin{theo}\label{theo2}
Suppose $M\subset \mathbb R^{n+1}$ is a closed $k$-convex
hypersurface satisfying the curvature equation (\ref{1.1}) with $2k>n$ for some
positive function $\psi(X, \nu)\in C^{2}(\Gamma)$, where $\Gamma$ is
an open neighborhood of the unit normal bundle of $M$ in $\mathbb
R^{n+1} \times \mathbb S^n$. Assume Conjecture \ref{con} holds,
then there is a constant $C$ depending only on $n, k$,
$\|M\|_{C^1}$, $\inf \psi$ and $\|\psi\|_{C^2}$, such that
 \begin{equation}\label{Mc2}
 \max_{X\in M, i=1,\cdots, n} \kappa_i(X) \le C.\end{equation}
\end{theo}

If one would like to derive global curvature estimate, the
inequality \eqref{s3.01} needs to be carefully studied. Thus, the
second result of this paper is to give a relatively explicit lower
bound of the left hand side of \eqref{s3.01}.

For any fixed indices $1\leq a,b,c\leq n$, we always let
$$\sigma_k(\kappa|a)=\frac{\p \sigma_k(\kappa)}{\p \kappa_a}, \sigma_k(\kappa|ab)=\frac{\p^2 \sigma_k(\kappa)}{\p \kappa_a\p \kappa_b},\sigma_k(\kappa|ab)=\frac{\p^3 \sigma_k(\kappa)}{\p \kappa_a\p \kappa_b\p \kappa_c}. $$
Let $\xi=(\xi_1,\xi_2,\cdots,\xi_{n})\in\mathbb{R}^n$ be
an $n$-dimensional vector. Suppose $1\leq i\leq n$ is some given index. We define four
quadratic forms, $\textbf{A}_{k;i},
\textbf{B}_{k;i}, \textbf{C}_{k;i},\textbf{D}_{k;i}$:
 \begin{eqnarray}
 \textbf{A}_{k;i}&=&\sum_{j\neq i}\sigma_{k-2}^2(\kappa|ij)\xi_j^2+\sum_{p\neq q; p,q\neq i}\left[\sigma_{k-2}^2(\kappa|ipq)
-\sigma_{k-1}(\kappa|ipq)\sigma_{k-3}(\kappa|ipq)\right]\xi_p\xi_q\nonumber;\\
\textbf{B}_{k;i}&=&\sum_{j\neq i}2\sigma_{k-2}(\kappa|ij)\xi_j^2-\sum_{p\neq q; p,q\neq i}\sigma_{k-2}(\kappa|ipq)\xi_p\xi_q;\nonumber\\
\textbf{C}_{k;i}&=&\sum_{j\neq i}\left[\kappa_j^2\sigma_{k-2}^2(\kappa|ij)-2\sigma_{k}(\kappa|ij)\sigma_{k-2}(\kappa|ij)\right]\xi_j^2\nonumber\\
&&+\sum_{p,q\neq i, p\neq q}\left[\sigma_{k}(\kappa|ipq)\sigma_{k-2}(\kappa|ipq)-\sigma_{k-1}^2(\kappa|ipq)\right]\xi_p\xi_q;\nonumber\\
\textbf{D}_{k;i}&=&\sum_{j\neq
i}\sigma_{k-1}^2(\kappa|ij)\xi_j^2+\sum_{p\neq q; p,q\neq i,
}\sigma_{k-1}(\kappa|ip)\sigma_{k-1}(\kappa|iq)\xi_p\xi_q.\nonumber
 \end{eqnarray}
For any positive constant $K$, we define
\begin{eqnarray}\label{ckK}
c_{k,K}=\frac{1}{K\kappa_i\sigma_{k-1}(\kappa|i)-1}.
\end{eqnarray}
Indeed, $c_{k,K}$ should be very small, if we let $K$ be sufficiently large, which we will detailed explain in the next section.       
Using the above notations, one has:

\begin{theo}\label{lemm17}
Assume $\kappa=(\kappa_1,\cdots,\kappa_n)\in\Gamma_{k}$,
$\kappa_1$ is the maximum entry of $\kappa$ and $\sigma_{k}(\kappa)$
has a positive lower bound $\sigma_{k}(\kappa)\geq N_0$. Then for
any given index $1\leq i\leq n$, if
$\kappa_i>\kappa_1-\sqrt{\kappa_1}/n$, for any $n$ dimensional vector $\xi=(\xi_1,\cdots,\xi_n)\in\mathbb{R}^n$, we have
\begin{eqnarray}\label{ex1}
&&\kappa_i\Big[K\Big(\sum_j\sigma_{k}^{jj}(\kappa)\xi_j\Big)^2-\sigma_{k}^{pp,qq}(\kappa)\xi_{p}\xi_{q}\Big]-\sigma^{ii}_{k}(\kappa)\xi_{i}^2+\sum_{j\neq
i}a_j\xi_{j}^2\\
&\geq&\frac{1}{c_{k,K}}\left[\kappa_i^2\textbf{A}_{k;i}+\sigma_{k}(\kappa)\textbf{B}_{k;i}+\textbf{C}_{k;i}-c_{k,K}\textbf{D}_{k;i}\right],\nonumber
\end{eqnarray}
when $\kappa_1$ and $K$ both are sufficiently large. Here $a_j$ and $c_{k,K}$ are defined by \eqref{aj} and \eqref{ckK}.
\end{theo}

The organization of this paper is as follows. In section 2, we will list the notations and lemmas needed in our proof. Section 3 will prove Theorem \ref{theo2}. Section 4 will prove Theorem \ref{lemm17}.

\section{Preliminary}
The operator $\sigma_k(\kappa)$ for
$\kappa=(\kappa_1,\kappa_2,\cdots,\kappa_n)\in\mathbb{R}^n$ has been
defined by
\begin{align*}
\sigma_k(\kappa)=\sum_{1\leq i_1<\cdots<i_k\leq n}\kappa_{i_1}\cdots
\kappa_{i_k}.
\end{align*}
Korevaar \cite{Kor} has shown that the cone $\Gamma_k$ also can be
characterized as
\begin{eqnarray}\label{Gamma}
&\left\{\kappa\in
\mathbb{R}^n;\sigma_k(\kappa)>0,\frac{\partial\sigma_k(\kappa)}{\partial\kappa_{i_1}}>0,
\cdots,\frac{\partial^k\sigma_k(\kappa)}{\partial\kappa_{i_1}\cdots\partial
\kappa_{i_k}}>0, \text{ for all } 1\leq i_{1}<\cdots<i_{k}\leq
n\right\}.\nonumber
\end{eqnarray}
Suppose $\kappa_1\geq\cdots\geq\kappa_n$, then using the above fact, we have
\begin{equation}\label{Gamma1}
\kappa_k+\kappa_{k+1}+\cdots+\kappa_n> 0\quad {\rm for}\quad
\kappa\in\Gamma_k.
\end{equation}
Thus, if $\kappa\in\Gamma_k$, the number of possible negative entries of $\kappa$ is at most $n-k$.

Let $\kappa(A)$ be
the eigenvalue vector of a matrix $A=(a_{ij})$.
Suppose $F$ is a function defined on the set of symmetric matrices. We let
$$f\left(\kappa (A)\right)=F(A).$$ Thus, we denote
$$F^{pq}=\frac{\p F}{\p a_{pq}}, \text{ and  } F^{pq,rs}=\frac{\p^2 F}{\p a_{pq}\p a_{rs}}.$$
For a local orthonormal frame, if $A$ is diagonal at a point, then
at this point, we have
$$F^{pp}=\frac{\p f}{\p \kappa_p}=f_p, \text{ and }  F^{pp,qq}=\frac{\p^2 f}{\p \kappa_p\p \kappa_q}=f_{pq}.$$

Thus the definition of the $k$-th elementary symmetric function can be extended to symmetric matrices. Suppose $W$ is an $n\times n$ symmetric matrix and $\kappa(W)$ is its eigenvalue vector. We define
$$\sigma_k(W)=\sigma_k(\kappa(W)),$$ which is the summation of the $k$-th principal minors of the matrix $W$.

Now we will list some algebraic identities and properties of
$\sigma_k$. In this paper, we will denote
$(\kappa|a)=(\kappa_1,\cdots,\kappa_{a-1},\kappa_{a+1},\cdots,\kappa_n)$.
For any $1\leq l\leq n$, the notation
$\sigma_l(\kappa|ab\cdots)$ means $\sigma_l((\kappa|ab\cdots))$.
Thus, we define

\par
\noindent (i) $\sigma^{pp}_k(\kappa):=\dfrac{\partial
\sigma_k(\kappa)}{\partial \kappa_p}=\sigma_{k-1}(\kappa|p)$ for any
given index $p=1,\cdots,n$;
\par
\noindent (ii) $\sigma^{pp,qq}_k(\kappa):=\dfrac{\partial^2
\sigma_k(\kappa)}{\partial \kappa_p\partial
\kappa_q}=\sigma_{k-2}(\kappa|pq)$  for any given indices
$p,q=1,\cdots,n$ and $\sigma^{pp,pp}_k(\kappa)=0$.

 Using the above
definitions, we have
\par
\noindent (iii)
$\sigma_k(\kappa)=\kappa_i\sigma_{k-1}(\kappa|i)+\sigma_k(\kappa|i)$
for any given index $i$;
\par
\noindent (iv)
$\dsum_{i=1}^n\kappa_i\sigma_{k-1}(\kappa|i)=k\sigma_k(\kappa)$.

 Thus, for a Codazzi  tensor $W=(w_{ij})$, we have

\noindent (v)
$-\dsum_{p,q,r,s}\sigma^{pq,rs}_k(W)w_{pql}w_{rsl}=\sum_{p,q}\sigma_k^{pp,qq}(W)w_{pql}^2-\sum_{p,q}\sigma^{pp,qq}_k(W)w_{ppl}w_{qql}$,\\
\noindent where $w_{pql}$ means the covariant derivative of $w_{pq}$
with respect to $l$ and $\sigma_k^{pq,rs}(W)=\dfrac{\p^2
\sigma_k(W)}{\p w_{pq}\p w_{rs}}$.  The meaning of Codazzi tensors
can be found in \cite{GRW}.

For $\kappa\in \Gamma_k$, suppose $\kappa_1\geq\cdots\geq\kappa_n$,
then we have
\par
\noindent (vi) $\sigma_{k-1}(\kappa|n)\geq \cdots\geq
\sigma_{k-1}(\kappa|1)>0$.\\
More details about the proof of these formulas can be found in
\cite{HS} and \cite{Wxj}.

 For $\kappa\in\mathbb{R}^n$, we have the
famous Maclaurin's inequality.
\par
\noindent (vii) $
\Big[\dfrac{\sigma_k(\kappa)}{C_n^k}\Big]^{1/k}\leq\Big[\dfrac{\sigma_l(\kappa)}{C_n^l}\Big]^{1/l}
$ for $k\geq l\geq 1, \kappa\in\Gamma_k$. (Maclaurin's inequality)
\par\noindent
Here $C^k_n$ is the combinational number, namely
$C^k_n=\dfrac{n!}{k!(n-k)!}$.

\par
Now, we list several lemmas frequently used in the other sections.
\begin{lemm} \label{Guan}
Assume that $k>l$, $W=(w_{ij})$ is a Codazzi tensor which is in
$\Gamma_k$. Denote $\al=\dfrac{1}{k-l}$.  Then, for $h=1,\cdots, n$
and any $\delta>0$, we have the following inequality
\begin{eqnarray}\label{s2.03}
&&-\sum_{p,q}\sigma_k^{pp,qq}(W)w_{pph}w_{qqh} +\left(1-\al+\dfrac{\al}{\delta}\right)\dfrac{(\sigma_k(W))_h^2}{\sigma_k(W)}\\
&\geq &\sigma_k(W)(\al+1-\delta\al)
\abc{\dfrac{(\sigma_l(W))_h}{\sigma_l(W)}}^2
-\dfrac{\sigma_k}{\sigma_l}(W)\sum_{p,q}\sigma_l^{pp,qq}(W)w_{pph}w_{qqh}.\nonumber
\end{eqnarray}
\end{lemm}
The proof can be found in \cite{GLL} and \cite{GRW}. Now we give another Lemma whose proof is
in \cite{LT}.
\begin{lemm} \label{lemm7}
Assume that $\kappa=(\kappa_1,\cdots,\kappa_n)\in\Gamma_{k}$. Then
for any given indices $1\leq i,j\leq n$, if $\kappa_i\geq\kappa_j$, we have
$$
|\sigma_{k-1}(\kappa|ij)|\leq \Theta\sigma_{k-1}(\kappa|j), \text{ where } \Theta=\sqrt{\dfrac{k(n-k)}{n-1}}.
$$
\end{lemm}

We also have
\begin{lemm}\label{lemm8}
Assume that $\kappa=(\kappa_1,\cdots,\kappa_n)\in \Gamma_k$ and $\kappa_1\geq  \cdots\geq \kappa_n$. Then for any
$0\leq s\leq k\leq n$, we have
\begin{eqnarray}\label{n2.5}
\dfrac{\kappa_1^s\sigma_{k-s}(\kappa)}{\sigma_k(\kappa)}\geq\dfrac{C_n^{k-s}}{C_n^k}.
\end{eqnarray}

\end{lemm}

\begin{proof}
Obviously, we have $\kappa_1>0$. Define
$\tilde\kappa=\dfrac{\kappa}{\kappa_1}=\left(1,\cdots,\dfrac{\kappa_n}{\kappa_1}\right)$.
Thus, we have $\dfrac{\sigma_k(\tilde\kappa)}{C_n^k}\leq 1$ and
$\tilde{\kappa}\in\Gamma_k$. By Maclaurin's inequality, we get
\begin{align*}
\dfrac{\sigma_{k-s}(\tilde\kappa)}{C_n^{k-s}}\geq
\Big[\dfrac{\sigma_k(\tilde\kappa)}{C_n^k}\Big]^{\frac{k-s}{k}}\geq
\dfrac{\sigma_k(\tilde\kappa)}{C_n^k},
\end{align*}
which implies
\begin{align*}
\dfrac{\kappa_1^s\sigma_{k-s}(\kappa)}{\sigma_k(\kappa)}=\dfrac{\sigma_{k-s}(\tilde\kappa)}{\sigma_{k}(\tilde\kappa)}\geq\dfrac{C_n^{k-s}}{C_n^k}.
\end{align*}

\end{proof}

Using the above lemma, we can prove
\begin{lemm}\label{lemm9}
Assume that $\kappa=(\kappa_1,\cdots,\kappa_n)\in\Gamma_{k}$ and $\kappa_1\geq  \cdots\geq \kappa_n$. Suppose any given indices $i,j$ satisfy $1\leq i,j\leq n$ and $i\neq j$.

(a) If $\kappa_i\leq 0$, then $-\kappa_i<\dfrac{(n-k)\kappa_1}{k}.$

(b) If $\kappa_i\leq\kappa_j\leq0$, then
$-(\kappa_i+\kappa_j)<\dfrac{2\sigma_{k}(\kappa|ij)}{\sigma_{k-1}(\kappa|ij)}.$
\end{lemm}
\begin{proof}
(a) Since $\kappa=(\kappa_1,\cdots,\kappa_n)\in\Gamma_{k}$, by
$$
\sigma_{k}(\kappa)=\kappa_i\sigma_{k-1}(\kappa|i)+\sigma_{k}(\kappa|i)>0,
\text{ and } \kappa_i\leq 0,
$$
we know that $\sigma_{k}(\kappa|i)>0$, which implies
$(\kappa|i)\in\Gamma_k$.
Applying Lemma \ref{lemm8} to $(\kappa|i)$ and using the above inequality, we get
$$
-\kappa_i<\dfrac{\sigma_{k}(\kappa|i)}{\sigma_{k-1}(\kappa|i)}\leq\dfrac{C_{n-1}^{k}\kappa_1}{C_{n-1}^{k-1}}=\frac{(n-k)\kappa_1}{k}.
$$

(b) Same as (a), using $\kappa_j\leq 0$, we know
$\sigma_{k}(\kappa|j)>0$. Thus, it is clear that
$$
\sigma_{k}(\kappa|j)=\kappa_i\sigma_{k-1}(\kappa|ij)+\sigma_{k}(\kappa|ij)>0,
$$
then, rewriting the above inequality, we have $
-\kappa_i<\dfrac{\sigma_{k}(\kappa|ij)}{\sigma_{k-1}(\kappa|ij)}.
$
Since $\kappa_i\leq\kappa_j\leq 0$, the above inequality implies
$-(\kappa_i+\kappa_j)<\dfrac{2\sigma_{k}(\kappa|ij)}{\sigma_{k-1}(\kappa|ij)}.$

\end{proof}

\begin{lemm}\label{lemm10}
Assume that $\kappa=(\kappa_1,\cdots,\kappa_n)\in\Gamma_k$, $1\leq
k\leq n$, and $\kappa_1\geq\cdots\geq\kappa_n$. Then for any $1\leq
s<k$, we have
 $$\sigma_{s}(\kappa)>\kappa_1\cdots\kappa_{s}.$$
\end{lemm}
\begin{proof}
Using $\kappa\in\Gamma_k\subset\Gamma_s$, we have
$$\kappa_1>0, \kappa_{2}>0,\cdots,\kappa_{s}>0,$$ and
$$\sigma_s(\kappa|1)> 0, \sigma_{s-1}(\kappa|12)> 0, \cdots, \sigma_1(\kappa|12\cdots s)> 0 .$$
Using the above inequalities, we get
\begin{align*}
\sigma_{s}(\kappa)=&\kappa_1\sigma_{s-1}(\kappa|1)+\sigma_s(\kappa|1)\\
> &\kappa_1\sigma_{s-1}(\kappa|1)=\kappa_1\kappa_{2}\sigma_{s-2}(\kappa|12)+\kappa_1\sigma_{s-1}(\kappa|12)\\
> &\kappa_1\kappa_{2}\sigma_{s-2}(\kappa|12)=\cdots\\
> &\kappa_1\cdots\kappa_{s}.
\end{align*}
\end{proof}

\begin{lemm}\label{lemm11}
Assume that $\kappa=(\kappa_1,\cdots,\kappa_n)\in\Gamma_k$, $1\leq
k\leq n$, and $\kappa_1\geq\cdots\geq\kappa_n$. For
any given indices $1\leq j\leq k$, there exists a positive constant $\theta$
only depending on $n,k$ such that
\begin{align*}
\sigma_{k}^{jj}(\kappa)\geq\dfrac{\theta\sigma_k(\kappa)}{\kappa_j}.
\end{align*}
Especially, we have $\kappa_1\sigma_{k}(\kappa|1)\geq \theta\sigma_k(\kappa)$.
\end{lemm}
\begin{proof}
We note that $\kappa_j>0$. We divide into two cases to prove our
Lemma.
\par
(a) If we have $\sigma_k(\kappa|j)\leq 0$, we easily see that
\begin{align*}
\sigma_{k}^{jj}(\kappa)=\dfrac{\sigma_k(\kappa)-\sigma_k(\kappa|j)}{\kappa_j}\geq\dfrac{\sigma_k(\kappa)}{\kappa_j}.
\end{align*}
\par
(b) If we have  $\sigma_k(\kappa|j)> 0$, using
$\kappa\in\Gamma_k$, we have $(\kappa|j)\in\Gamma_k$.
Thus, applying Lemma \ref{lemm10} to $(\kappa|j)$, we get
\begin{align*}
\sigma_{k}^{jj}(\kappa)=\sigma_{k-1}(\kappa|j)\geq
\dfrac{\kappa_{1}\cdots\kappa_k}{\kappa_j}.
\end{align*}
In view of \eqref{Gamma1}, for any $j>k$, we have $|\kappa_j|\leq n\kappa_k$. Thus, there exists some constant $\theta$ only depending on $n,k$ such that
$\theta\sigma_k(\kappa)\leq \kappa_1\cdots\kappa_k$. Therefore, we obtain our lemma.

\end{proof}

It is easy to see that if $\kappa_i\geq
\kappa_1-\sqrt{\kappa_1}/n$, we have
$$\kappa_i\sigma_{k-1}(\kappa|i)\geq
\kappa_1\sigma_{k-1}(\kappa|1)/2\geq\theta\sigma_k(\kappa)/2>0,$$
when $\kappa_1$ is sufficiently large. Here $\theta$ is the constant
given in Lemma \ref{lemm11}. Therefore, $c_{k,K}$ defined by \eqref{ckK} is a positive
constant and can be very small if the constant $K$ is sufficiently
large. Thus, throughout the paper, we always assume $K$ is
sufficiently large and then $c_{k,K}$ is positive.

\section{The global  curvature estimates }
\par
In this section, we will derive the global $C^2$-estimates for the
 curvature equation (\ref{1.1}) based on Conjecture \ref{con}, namely, to prove Theorem \ref{theo2}.  

\par
Denote $X,\nu$ to be the position vector and outer normal vector of
$M$. Set $u(X)=\langle X, \nu(X)\rangle$, where
$\langle\cdot,\cdot\rangle$ denotes the inner product of the ambient
space. By the assumption that $M$ is a starshaped hypersurface with
a $C^1$ bound, $u$ is bounded from below and above by two positive
constants. At every point in the hypersurface $M$, choose a local
coordinate frame $\{ \p/(\p x_1),\cdots,\p/(\p x_{n+1})\}$ in
$\mathbb{R}^n$ such that the first $n$ vectors are the local
coordinates of the hypersurface and the last one is the unit outer
normal vector. 

We let
$h_{ij}$  be the second fundamental form of the hypersurface $M$.  The following
geometric formulas are well known (e.g., \cite{GLL}),

\begin{equation}
h_{ij}=\langle\partial_iX,\partial_j\nu\rangle,
\end{equation}
and
\begin{equation}
\begin{array}{rll}
X_{ij}=& -h_{ij}\nu\quad {\rm (Gauss\ formula)}\\
(\nu)_i=&h_{ij}\partial_j\quad {\rm (Weigarten\ equation)}\\
h_{ijk}=& h_{ikj}\quad {\rm (Codazzi\ formula)}\\
R_{ijkl}=&h_{ik}h_{jl}-h_{il}h_{jk}\quad {\rm (Gauss\ equation)},\\
\end{array}
\end{equation}
where $R_{ijkl}$ is the $(4,0)$-Riemannian curvature tensor. We also
have
\begin{equation}\label{4th}
\begin{array}{rll}
h_{ijkl}=& h_{ijlk}+h_{mj}R_{imlk}+h_{im}R_{jmlk}\\
=& h_{klij}+(h_{mj}h_{il}-h_{ml}h_{ij})h_{mk}+(h_{mj}h_{kl}-h_{ml}h_{kj})h_{mi}.\\
\end{array}
\end{equation}

\par
For the function $u$, we consider the following test function
$$
\phi=\log\log P-N\ln u.
$$
Here $N$ is some undetermined constant and the function $P$ is defined by $$P=\dsum_le^{\kappa_l}. $$

We may assume that the maximum of $\phi$ is achieved  at some point
$X_0\in M$. By a proper rotation of the coordinates, we may assume the matrix
$(h_{ij})$ is diagonal at that point, and we can further  assume that
$h_{11}\geq h_{22}\cdots\geq h_{nn}$. Since $\kappa_1,\kappa_2\cdots,\kappa_n$ denote the principal curvatures of $M$, then we have $\kappa_i=h_{ii}$.
\par
Covariant differentiating the function $\phi$ twice  at  $X_0$, we have
\begin{equation}\label{e2.11}
\phi_i = \dfrac{P_i}{P\log P}- N\frac{h_{ii}\langle
X,\p_i\rangle}{u} = 0,
\end{equation}
and
\begin{align*}
\phi_{ii}
=& \frac{P_{ii}}{P\log P}-\frac{P_i^2}{P^2\log P}-\frac{P_i^2}{(P\log P)^2}- \frac{N}{u}\sum_lh_{il,i}\langle \p_l,X \rangle-\frac{N h_{ii}}{u}\nonumber\\
&+Nh_{ii}^2+N\frac{h_{ii}^2\langle X,\p_i\rangle^2}{u^2}.
\end{align*}
Here we have 
\begin{align*}
P_i=\dsum_le^{\kappa_l}h_{lli},\quad
P_{ii}=\sum_le^{\kappa_l}h_{llii}+\sum_le^{\kappa_l}h_{lli}^2+\sum_{\alpha\neq
\beta}\frac{e^{\kappa_{\alpha}}-e^{\kappa_{\beta}}}{\kappa_{\alpha}-\kappa_{\beta}}h_{\alpha\beta
i}^2,
\end{align*}
and at $X_0$,
\begin{align*}
h_{llii}=&h_{ii,ll}+h_{ii}h_{ll}^2-h_{ii}^2h_{ll}.
\end{align*}
Then, we have
\begin{align*}
\phi_{ii}
\geq&\frac{1}{P\log P}\bigg[\sum_le^{\kappa_l}h_{ii,ll}+\sum_le^{\kappa_l}(h_{ii}h_{ll}^2-h_{ii}^2h_{ll})\nonumber\\
&+\sum_le^{\kappa_l}h_{lli}^2+\sum_{\alpha\neq \beta}\frac{e^{\kappa_{\alpha}}-e^{\kappa_{\beta}}}{\kappa_{\alpha}-\kappa_{\beta}}h_{\alpha\beta i}^2-\Big(\frac{1}{P}+\frac{1}{P\log P}\Big)P_i^2\bigg]\nonumber\\
&- \frac{N\sum_lh_{iil}\langle \p_l,X \rangle}{u}-\frac{
Nh_{ii}}{u}+Nh_{ii}^2.
\end{align*}

Contracting with $\sigma_{k}^{ii}$, we have
\begin{align}\label{e2.13}
\sigma_{k}^{ii}\phi_{ii}\geq&\frac{1}{P\log
P}\bigg[\sum_le^{\kappa_l}\sigma_{k}^{ii}h_{ii,ll}+k\psi\sum_le^{\kappa_l}h_{ll}^2-\sigma_{k}^{ii}h_{ii}^2\sum_le^{\kappa_l}h_{ll}
\\
&+\sum_l\sigma_{k}^{ii}e^{\kappa_l}h_{lli}^2+\sum_{\alpha\neq \beta}\sigma_{k}^{ii}\frac{e^{\kappa_{\alpha}}-e^{\kappa_{\beta}}}{\kappa_{\alpha}-\kappa_{\beta}}h_{\alpha\beta i}^2-\Big(\frac{1}{P}+\frac{1}{P\log P}\Big)\sigma_{k}^{ii}P_i^2\bigg]\nonumber\\
&- \frac{N\sum_l\sigma_{k}^{ii}h_{iil}\langle \p_l,X
\rangle}{u}-\frac{ Nk\psi}{u}+N\sigma_{k}^{ii}h_{ii}^2.\nonumber
\end{align}\par
At $X_0$, differentiating the equation (\ref{1.1}) twice, we have
\begin{align}\label{e2.14}
\sigma_{k}^{ii}h_{iil}=d_X\psi(\p_l)+h_{ll}d_{\nu}\psi(\p_l),
\end{align}
and
\begin{align}\label{e2.15}
\sigma_{k}^{ii}h_{iill}+\sigma_{k}^{pq,rs}h_{pql}h_{rsl}\geq
-C-Ch_{11}^2+\sum_sh_{sll}d_{\nu}\psi(\p_s),
\end{align}
where $C$ is some uniform constant.

Inserting  (\ref{e2.15}) into (\ref{e2.13}), we have
\begin{align}
\label{e2.16}
\sigma_{k}^{ii}\phi_{ii}\geq &\frac{1}{P\log P}\bigg[\sum_le^{\kappa_l}(-C-Ch_{11}^2-\sigma_{k}^{pq,rs}h_{pql}h_{rsl})+\sum_{l,s}e^{\kappa_l}h_{sll}d_{\nu}\psi(\p_s)\\
&+k\psi\sum_le^{\kappa_l}h_{ll}^2-\sigma_{k}^{ii}h_{ii}^2\sum_le^{\kappa_l}h_{ll}+\sum_l\sigma_{k}^{ii}e^{\kappa_l}h_{lli}^2\nonumber\\
&+\sum_{\alpha\neq \beta}\sigma_{k}^{ii}\frac{e^{\kappa_{\alpha}}-e^{\kappa_{\beta}}}{\kappa_{\alpha}-\kappa_{\beta}}h_{\alpha\beta i}^2-\Big(\frac{1}{P}+\frac{1}{P\log P}\Big)\sigma_{k}^{ii}P_i^2\bigg]\nonumber\\
&- \frac{N\sum_l\sigma_{k}^{ii}h_{iil}\langle \p_l,X
\rangle}{u}-\dfrac{Nk\psi}{u}+N\sigma_{k}^{ii}h_{ii}^2.\nonumber
\end{align}
By (\ref{e2.11}) and (\ref{e2.14}), we have
\begin{align}\label{e2.17}
\frac{1}{P\log
P}\sum_{l,s}e^{\kappa_l}h_{sll}d_{\nu}\psi(\p_s)-\frac{N}{u}\sum_l\sigma_{k}^{ii}h_{iil}\langle
\p_l, X\rangle =&-\frac{N}{u}\sum_ld_X\psi(\p_l)\langle
X,\p_l\rangle.
\end{align}

Denote
\begin{eqnarray}
&&A_i=e^{\kappa_i}\Big[K(\sigma_{k})_i^2-\sum_{p\neq
q}\sigma_{k}^{pp,qq}h_{ppi}h_{qqi}\Big], \ \  B_i=2\sum_{l\neq
i}\sigma_{k}^{ii,ll}e^{\kappa_l}h_{lli}^2, \nonumber \\
&&C_i=\sigma_{k}^{ii}\sum_le^{\kappa_l}h_{lli}^2; \  \
D_i=2\sum_{l\neq
i}\sigma_{k}^{ll}\frac{e^{\kappa_l}-e^{\kappa_i}}{\kappa_l-\kappa_i}h_{lli}^2,
\ \ E_i=\frac{1+\log P}{P\log P}\sigma_{k}^{ii}P_i^2\nonumber.
\end{eqnarray}
Using
$$
-\dsum_l\sigma_{k}^{pq,rs}h_{pql}h_{rsl}=\dsum_{p\neq
q}\sigma_{k}^{pp,qq}h_{pql}^2-\dsum_{p\neq
q}\sigma_{k}^{pp,qq}h_{ppl}h_{qql},
$$
and (\ref{e2.16}), for any $K>1$, we have
\begin{align}\label{e2.18}
\sigma_{k}^{ii}\phi_{ii}
\geq &\frac{1}{P\log P}\bigg[\sum_le^{\kappa_l}\Big(K(\sigma_{k})_l^2-\dsum_{p\neq q}\sigma_{k}^{pp,qq}h_{ppl}h_{qql}+\dsum_{p\neq q}\sigma_{k}^{pp,qq}h_{pql}^2\Big)\\
&+\sum_l\sigma_{k}^{ii}e^{\kappa_l}h_{lli}^2+\sum_{\alpha\neq
\beta}\sigma_{k}^{ii}\frac{e^{\kappa_{\alpha}}-e^{\kappa_{\beta}}}{\kappa_{\alpha}-\kappa_{\beta}}h_{\alpha\beta
i}^2-\frac{1+\log P}{P\log
P}\sigma_{k}^{ii}P_i^2\nonumber\\
&-CP-CKPh_{11}^2\bigg]+(N-1)\sigma_{k}^{ii}h_{ii}^2\nonumber\\
\geq &\frac{1}{P\log
P}\dsum_i(A_i+B_i+C_i+D_i-E_i)+(N-1)\sigma_{k}^{ii}h_{ii}^2-\dfrac{C+CKh_{11}^2}{\log
P}\nonumber.
\end{align}
\par
We claim that
\begin{equation}\label{e3.16}
A_i+B_i+C_i+D_i-E_i\geq 0
\end{equation}
for all $i=1,\cdots,n$. Therefore by (\ref{e2.18}), we obtain
\begin{align*}
0\geq&\sigma_{k}^{ii}\phi_{ii}
\geq \frac{1}{P\log P}\dsum_i(A_i+B_i+C_i+D_i-E_i)\\
&+(N-1)\sigma_{k}^{ii}h_{ii}^2-\dfrac{C+CKh_{11}^2}{\log P}\nonumber\\
\geq&(N-1)\theta h_{11}-\dfrac{C+CKh_{11}^2}{\log P}\nonumber.
\end{align*}
Here we have used Lemma \ref{lemm11}. Choosing a sufficiently large
positive constant $N$, we obtain an upper bound of $h_ {11}$.

\par
Next, we will divide into two cases to prove our claim
(\ref{e3.16}).

\noindent (I) $\kappa_i\leq \kappa_1-\sqrt{\kappa_1}/n $;

\noindent (II) $\kappa_i>  \kappa_1-\sqrt{\kappa_1}/n$.

At first, we need to prove  the following Lemma.
\par

\par
\begin{lemm} \label{lemm13}    
Assume $\kappa=(\kappa_1,\cdots,\kappa_n)\in\Gamma_k$, $2k>n$, and
$\kappa_1$ is the maximum entry of $\kappa$. Denote
$\delta_k=\dfrac{1}{3k}$, then we have
\begin{align}\label{s3.03}
(2-\delta_k)e^{\kappa_l}\sigma_{k-2}(\kappa|il)+(2-\delta_k)\dfrac{e^{\kappa_l}-e^{\kappa_i}}{\kappa_l-\kappa_i}
\sigma_{k-1}(\kappa|l)\geq
\dfrac{e^{\kappa_l}}{\kappa_1}\sigma_{k-1}(\kappa|i)
\end{align}
for all indices $i,l$ satisfying $l\neq i$, if $\kappa_1$ is
sufficiently large.
\end{lemm}
\begin{proof}  Obviously we have the following identity,
$$
\sigma_{k-1}(\kappa|l)=\sigma_{k-1}(\kappa|i)+(\kappa_i-\kappa_l)\sigma_{k-2}(\kappa|il).
$$
Multiplying $\dfrac{e^{\kappa_l}-e^{\kappa_i}}{\kappa_l-\kappa_i}$
in  both sides of the above identity, we have
\begin{align}\label{s3.04}
e^{\kappa_l}\sigma_{k-2}(\kappa|il)+\dfrac{e^{\kappa_l}-e^{\kappa_i}}{\kappa_l-\kappa_i}
\sigma_{k-1}(\kappa|l)
=e^{\kappa_i}\sigma_{k-2}(\kappa|il)+\dfrac{e^{\kappa_l}-e^{\kappa_i}}{\kappa_l-\kappa_i}
\sigma_{k-1}(\kappa|i).
\end{align}
Using (\ref{s3.04}), in order to prove \eqref{s3.03}, we only need
to show
\begin{eqnarray}\label{n3.5}
(2-\delta_k)\dfrac{e^{\kappa_l}-e^{\kappa_i}}{\kappa_l-\kappa_i}
\geq\dfrac{e^{\kappa_l}}{\kappa_1},
\end{eqnarray}
which we will divide into four cases to prove.
\par
\noindent Case (a): Suppose $\kappa_l\leq \kappa_i$. We have
\begin{align*}
\dfrac{e^{\kappa_l}-e^{\kappa_i}}{\kappa_l-\kappa_i}
=e^{\kappa_l}\dfrac{e^{\kappa_i-\kappa_l}-1}{\kappa_i-\kappa_l} \geq
e^{\kappa_l}\geq\dfrac{e^{\kappa_l}}{\kappa_1},
\end{align*}
if $\kappa_1$ is sufficiently large. Here we have used the
inequality $e^t>1+t$ for $t>0$.
\par
\noindent Case (b): Suppose $0<\kappa_l-\kappa_i\leq 1$.  By the
mean value theorem, there exists some constant $\xi$ satisfying
$\kappa_i<\xi<\kappa_l$, such that
\begin{align*}
\dfrac{e^{\kappa_l}-e^{\kappa_i}}{\kappa_l-\kappa_i} =e^{\xi}\geq
e^{\kappa_i}\geq e^{\kappa_l-1} \geq\dfrac{e^{\kappa_l}}{\kappa_1},
\end{align*}
if $\kappa_1$ is sufficiently large.
\par
\noindent Case (c): Suppose $1<\kappa_l-\kappa_i\leq \kappa_1$. We
have
\begin{align*}
(2-\delta_k)\dfrac{e^{\kappa_l}-e^{\kappa_i}}{\kappa_l-\kappa_i}
  \geq&
(2-\delta_k)e^{\kappa_l}\dfrac{1-e^{-1}}{\kappa_l-\kappa_i}
 \geq
(2-\delta_k)(1-e^{-1})\dfrac{e^{\kappa_l}}{\kappa_1}
 \geq\dfrac{e^{\kappa_l}}{\kappa_1}.
\end{align*}
\par
\noindent Case (d): Suppose $\kappa_l-\kappa_i>\kappa_1$. In this
case, our condition implies  $\kappa_i<0$. By Lemma \ref{lemm9} and
$2k>n$, we know that
$-\kappa_i<\dfrac{n-k}{k}\kappa_1\leq\dfrac{k-1}{k}\kappa_1$, then
we have
$$\kappa_l-\kappa_i\leq\kappa_1-\kappa_i<\dfrac{2k-1}{k}\kappa_1.$$
Thus, in this case,
\begin{align*}
(2-\delta_k)\dfrac{e^{\kappa_l}-e^{\kappa_i}}{\kappa_l-\kappa_i}
  \geq&
(2-\delta_k)e^{\kappa_l}\dfrac{1-e^{-\kappa_1}}{\kappa_l-\kappa_i}
 \geq
\dfrac{(2-\delta_k)(1-e^{-\kappa_1})}{(2k-1)/k}\dfrac{e^{\kappa_l}}{\kappa_1}.
\end{align*}
We obviously have
$$\dfrac{(2-\delta_k)}{(2k-1)/k}> 1+\dfrac{1}{3k}.$$
Thus we get
 $\dfrac{(2-\delta_k)(1-e^{-\kappa_1})}{(2k-1)/k}\geq 1$ if
$\kappa_1$ is sufficiently large, which gives the desired
inequality.
\end{proof}

Next lemma will handle the case (I).

\begin{lemm}\label{le2}
Assume $\kappa\in\Gamma_k$, $2k>n$, and $\kappa_1$ is the maximum
entry of $\kappa$. For given index $1\leq i\leq n$, if $\kappa_i\leq
\kappa_1-\sqrt{\kappa_1}/n$ then we have
\begin{align*}
A_i+B_i+C_i+D_i-E_i\geq0,
\end{align*}
when the constant $K$ and the biggest eigenvalue $\kappa_1$ both are sufficiently large.
\end{lemm}
\begin{proof}  Firstly, by Lemma 2.2 of \cite{GRW}, we have $A_i> 0$,
if the constant $K$ is sufficiently large.
By the Cauchy-Schwarz inequality, we have
\begin{align}\label{e2.19}
P_i^2= & e^{2\kappa_i}h_{iii}^2 +2 \dsum_{l\neq
i}e^{\kappa_i+\kappa_l}h_{iii}h_{lli}
+ \Big(\dsum_{l\neq i}e^{\kappa_l}h_{lli}\Big)^2\\
\leq & e^{2\kappa_i}h_{iii}^2 +2 \dsum_{l\neq
i}e^{\kappa_i+\kappa_l}h_{iii}h_{lli}+ (P-e^{\kappa_i})\dsum_{l\neq
i}e^{\kappa_l}h_{lli}^2. \nonumber
\end{align}
Using (\ref{e2.19}),  we have
\begin{align}\label{e2.20}
&B_i+C_i+D_i-E_i\\
\geq&2\dsum_{l\neq i}e^{\kappa_l}\sigma_{k}^{ll,ii}h_{lli}^2 +
2\dsum_{l\neq
i}\dfrac{e^{\kappa_l}-e^{\kappa_i}}{\kappa_l-\kappa_i}\sigma_{k}^{ll}h_{lli}^2
-\dfrac{1}{\log P}\dsum_{l\neq
i}e^{\kappa_l}\sigma_{k}^{ii}h_{lli}^2\nonumber\\
&+\frac{1+\log P}{P\log P}\dsum_{l\neq
i}e^{\kappa_l+\kappa_i}\sigma_{k}^{ii}h_{lli}^2+e^{\kappa_i}\sigma_{k}^{ii}h_{iii}^2 \nonumber\\
&-\frac{1+\log P}{P\log P}e^{2\kappa_i}\sigma_{k}^{ii}h_{iii}^2
-2\frac{1+\log P}{P\log P}\dsum_{l\neq
i}e^{\kappa_i+\kappa_l}\sigma_{k}^{ii}h_{iii}h_{lli} \nonumber.
\end{align}
Note that $\log P>\kappa_1$, using Lemma \ref{lemm13}, we have
\begin{align}\label{e2.21}
(2-\dfrac{1}{3k})\dsum_{l\neq
i}e^{\kappa_l}\sigma_{k}^{ll,ii}h_{lli}^2 +
(2-\dfrac{1}{3k})\dsum_{l\neq
i}\dfrac{e^{\kappa_l}-e^{\kappa_i}}{\kappa_l-\kappa_i}\sigma_{k}^{ll}h_{lli}^2
-\dfrac{1}{\log P}\dsum_{l\neq
i}e^{\kappa_l}\sigma_{k}^{ii}h_{lli}^2\geq 0.
\end{align}
On the other hand, it is clear that
\begin{align}\label{CauS}
&\dsum_{l\neq i,1}e^{\kappa_l+\kappa_i}\sigma_{k}^{ii}h_{lli}^2
-2\dsum_{l\neq
i,1}e^{\kappa_i+\kappa_l}\sigma_{k}^{ii}h_{iii}h_{lli}\geq-\dsum_{l\neq
i,1}e^{\kappa_l+\kappa_i}\sigma_{k}^{ii}h_{iii}^2.
\end{align}
Then, using the above two inequalities, \eqref{e2.20} becomes
\begin{align}\label{e2.23}
&B_i+C_i+D_i-E_i\\
\geq&\frac{1+\log P}{P\log
P}e^{\kappa_1+\kappa_i}\sigma_{k}^{ii}h_{11i}^2+e^{\kappa_i}\sigma_{k}^{ii}h_{iii}^2 \nonumber\\
&-\frac{1+\log P}{P\log P}\dsum_{l\neq
1}e^{\kappa_l+\kappa_i}\sigma_{k}^{ii}h_{iii}^2
-2\frac{1+\log P}{P\log P}e^{\kappa_i+\kappa_1}\sigma_{k}^{ii}h_{iii}h_{11i} \nonumber\\
&+\dfrac{1}{3k}e^{\kappa_1}\sigma_{k}^{11,ii}h_{11i}^2 +
\dfrac{1}{3k}\dfrac{e^{\kappa_1}-e^{\kappa_i}}{\kappa_1-\kappa_i}\sigma_{k}^{11}h_{11i}^2.
\nonumber
\end{align}

A straightforward calculation shows that
\begin{align}
e^{\kappa_i}\sigma_{k}^{ii}h_{iii}^2 -\frac{1+\log P}{P\log
P}\dsum_{l\neq 1}e^{\kappa_l+\kappa_i}\sigma_{k}^{ii}h_{iii}^2\geq&
\Big(\frac{e^{\kappa_1}}{P}-\frac{1}{\log
P}\Big)e^{\kappa_i}\sigma_{k}^{ii}h_{iii}^2 \nonumber\\
\geq& \frac{1}{n+1}e^{\kappa_i}\sigma_{k}^{ii}h_{iii}^2\nonumber
\end{align}
and
\begin{align}
-2\frac{1+\log P}{P\log
P}e^{\kappa_i+\kappa_1}\sigma_{k}^{ii}|h_{iii}h_{11i}| \geq &
-\frac{3}{P}e^{\kappa_i+\kappa_1}\sigma_{k}^{ii}|h_{iii}h_{11i}|\nonumber\\
\geq& -3e^{\kappa_i}\sigma_{k}^{ii}|h_{iii}h_{11i}|,\nonumber
\end{align}
hold at the same time if  $\kappa_1$ is sufficiently lagre. We let
$l=1$ in \eqref{s3.04} and we have
\begin{align}\label{e2.24}
e^{\kappa_1}\sigma_{k}^{11,ii}h_{11i}^2 +
\dfrac{e^{\kappa_1}-e^{\kappa_i}}{\kappa_1-\kappa_i}\sigma_{k}^{11}h_{11i}^2
=e^{\kappa_i}\sigma_{k}^{11,ii}h_{11i}^2 +
\dfrac{e^{\kappa_1}-e^{\kappa_i}}{\kappa_1-\kappa_i}\sigma_{k}^{ii}h_{11i}^2.
\end{align}
By Taylor's Theorem, we also have
\begin{align}\label{e2.25}
\dfrac{e^{\kappa_1}-e^{\kappa_i}}{\kappa_1-\kappa_i}\sigma_{k}^{ii}h_{11i}^2
=e^{\kappa_i}\dsum_{m\geq
1}\dfrac{(\kappa_1-\kappa_i)^{m-1}}{m!}\sigma_{k}^{ii}h_{11i}^2.
\end{align}
Combining the previous four formulas and using (\ref{e2.23}), we obtain
\begin{align}
&B_i+C_i+D_i-E_i\nonumber\\
\geq&e^{\kappa_i}\sigma_{k}^{ii}\Big[\frac{1}{n+1}h_{iii}^2-3|h_{iii}h_{11i}|
+\frac{1}{3k}\dsum_{m\geq
1}\dfrac{(\kappa_1-\kappa_i)^{m-1}}{m!}h_{11i}^2\Big] \nonumber\\
\geq& 0,\nonumber
\end{align}
if $\kappa_1$ is sufficiently large.
\end{proof}

For the case (II), we first prove that
\begin{lemm}\label{lemm16}
Assume $\kappa=(\kappa_1,\kappa_2,\cdots,\kappa_n)\in\Gamma_k$,
$n<2k$, $\kappa_1$ is the maximum entry of $\kappa$ and
$\sigma_k(\kappa)$ has a lower bound $\sigma_k(\kappa)\geq N_0>0$.
Then for any given indices $i,j$ satisfying $1\leq i,j\leq n$ and
$j\neq i$, if
 $\kappa_i>\kappa_1-\sqrt{\kappa_1}/n$, we have
\begin{align}\label{a3.6}
\dfrac{2\kappa_i(1-e^{\kappa_j-\kappa_i})}{\kappa_i-\kappa_j}\sigma_{k}^{jj}(\kappa)\geq
\sigma_{k}^{jj}(\kappa)+(\kappa_i+\kappa_j)\sigma_k^{ii,jj}(\kappa),
\end{align}
when $\kappa_1$ is sufficiently large.
\end{lemm}
\begin{proof}
If $\kappa_i=\kappa_j$, the left hand side of (\ref{a3.6}) should be
viewed as a limitation when $\kappa_j$ converging to $\kappa_i$,
about which we refer \cite{Ball} for more explanation. It is easy to
see that the limitation  is $2\kappa_i\sigma_k^{jj}(\kappa)$. Thus,
a straightforward calculation shows
\begin{align*}
&2\kappa_i\sigma_{k-1}(\kappa|j)-\sigma_{k-1}(\kappa|j)-(\kappa_i+\kappa_j)\sigma_{k-2}(\kappa|ij)\\
=&2\kappa_i\sigma_{k-1}(\kappa|j)-2\sigma_{k-1}(\kappa|j)-\sigma_{k-1}(\kappa|i)+2\sigma_{k-1}(\kappa|ij).
\end{align*}
Using Lemma \ref{lemm7}, $|\sigma_{k-1}(\kappa|ij)|$ can be bounded
by $\Theta\sigma_{k-1}(\kappa|j)$. Thus, since we have
$\sigma_{k-1}(\kappa|i)=\sigma_{k-1}(\kappa|j)$, the above formula
is positive if $\kappa_1$ is sufficiently large.

 If $\kappa_i\neq\kappa_j$, we have the following identity,
\begin{align}\label{n3.7}
\sigma_{k}^{jj}(\kappa)+(\kappa_i+\kappa_j)\sigma_k^{ii,jj}(\kappa)
=\dfrac{2\kappa_i}{\kappa_i-\kappa_j}\sigma_k^{jj}(\kappa)-\dfrac{\kappa_i+\kappa_j}{\kappa_i-\kappa_j}\sigma_k^{ii}(\kappa).
\end{align}
In view of \eqref{n3.7}, in order to prove \eqref{a3.6}, it suffices
to show
\begin{align}\label{nn}
\dfrac{-2\kappa_ie^{\kappa_j-\kappa_i}}{\kappa_i-\kappa_j}\sigma_{k}^{jj}(\kappa)\geq
-\dfrac{\kappa_i+\kappa_j}{\kappa_i-\kappa_j}\sigma_k^{ii}(\kappa).
\end{align}
Let's define some function:
\begin{eqnarray}\label{n3.8}
L=\left\{\begin{matrix}(\kappa_i+\kappa_j)e^{\kappa_i-\kappa_j}\sigma_k^{ii}(\kappa)-2\kappa_i\sigma_k^{jj}(\kappa)&
\kappa_i>\kappa_j,\\
2\kappa_ie^{\kappa_j-\kappa_i}\sigma_k^{jj}(\kappa)-(\kappa_i+\kappa_j)\sigma_k^{ii}(\kappa)&
\kappa_i<\kappa_j. \end{matrix}\right.
\end{eqnarray}
Obviously, $L\geq 0$ implies \eqref{nn}. Thus, let's prove $L\geq 0$
in the following for $\kappa_i>\kappa_j$ and $\kappa_i<\kappa_j$
respectively.

\vskip 1mm
\par
If $\kappa_i>\kappa_j$, we let $t=\kappa_i-\kappa_j$. Thus we have
$t> 0$. We divide into two cases to prove $L$ is non negative for
$\kappa_i>\kappa_j$.

Case (a): Suppose $t\geq\sqrt{\kappa_1}$. In this case, our
assumption gives $e^t\geq
e^{\sqrt{\kappa_1}}>\dfrac{(\sqrt{\kappa_1})^{2k+1}}{(2k+1)!}$. Here
we have used the Taylor expansion in the second inequality.

If $\kappa_j\leq 0$, using $n\leq 2k-1$ and Lemma \ref{lemm9}, we
have $-\kappa_j<\dfrac{(n-k)\kappa_1}{k}\leq
\dfrac{(k-1)\kappa_1}{k}$. Thus, since
$\kappa_i>\kappa_1-\sqrt{\kappa_1}/n$, we have
$\kappa_i+\kappa_j>\dfrac{\kappa_1}{2k}$ if $\kappa_1>10$. If
$\kappa_j>0$, it is easy to see
$\kappa_i+\kappa_j>\dfrac{\kappa_1}{2k}$. Thus, in any cases, we
have
$$L\geq \frac{\kappa_1^{k+\frac{3}{2}}}{2k (2k+1)!}\sigma_k^{ii}(\kappa)-2\kappa_1\sigma_k^{jj}(\kappa)\geq \kappa_1\left(\frac{\kappa_1^{k-1}\sqrt{\kappa_1}\theta N_0}{2k (2k+1)!}-2\sigma_k^{jj}(\kappa)\right)\geq 0,$$
if $\kappa_1$ is sufficiently large. Here we have used
$\kappa_1\sigma_k^{ii}(\kappa)\geq \theta\sigma_k(\kappa)$.

Case (b): Suppose $t<\sqrt{\kappa_1}$. Using
$\kappa_i>\kappa_1-\sqrt{\kappa_1}/n$, we have $\kappa_j>\kappa_1/2$
if $\kappa_1>10$. For simplification purpose, denote
$\tilde\sigma_m=\sigma_m(\kappa|ij)$. We have
\begin{align}\label{s3.07}
L=&(\kappa_i+\kappa_j)e^{t}\sigma_{k-1}(\kappa|i)-2\kappa_i\sigma_{k-1}(\kappa|j)\\
=&(\kappa_i+\kappa_j)e^{t}(\kappa_j\tilde\sigma_{k-2}+\tilde\sigma_{k-1})-2\kappa_i(\kappa_i\tilde\sigma_{k-2}+\tilde\sigma_{k-1})\nonumber\\
=&[\kappa_j(\kappa_j+\kappa_i)(e^{t}-1)-t(\kappa_j+\kappa_i)-t\kappa_i]\tilde\sigma_{k-2}+[(\kappa_i+\kappa_j)(e^t-1)-t]\tilde\sigma_{k-1},\nonumber
\end{align}
where in the last equality, we have used $t=\kappa_i-\kappa_j$. We
further divide into two sub-cases to prove the nonnegativity of $L$.

Subcase (b1): Suppose $\tilde\sigma_{k-1}\geq 0$. Note that
$e^t>1+t$. By (\ref{s3.07}) and
$\kappa_i>\kappa_1-\sqrt{\kappa_1}/n, \kappa_j>\kappa_1/2$, we get
$L\geq0$.

Subcase (b2): Suppose $\tilde\sigma_{k-1}< 0$. Inserting the
identity
\begin{align}\label{n311}
(\kappa_i+\kappa_j)\tilde\sigma_{k-1}=\sigma_k(\kappa)-\kappa_i\kappa_j\tilde\sigma_{k-2}-\tilde\sigma_{k}
\end{align}
into the last equality of (\ref{s3.07}), we get
\begin{align}\label{s3.08}
L=&[(\kappa_j^2+\kappa_j\kappa_i)(e^{t}-1)-t(\kappa_j+\kappa_i)-t\kappa_i]\tilde\sigma_{k-2}-t\tilde\sigma_{k-1}\\
&+(e^t-1)[\sigma_k(\kappa)-\kappa_i\kappa_j\tilde\sigma_{k-2}-\tilde\sigma_{k}]\nonumber\\
\geq&[\kappa_j^2(e^t-1)-3t\kappa_i]\tilde\sigma_{k-2}+(e^t-1)[\sigma_k(\kappa)-\tilde\sigma_{k}],\nonumber
\end{align}
where in the last inequality, we have used $\kappa_i\geq
\kappa_j>0,t>0$, and $\tilde\sigma_{k-1}<0$.

For $\sigma_{k-1}(\kappa|i)$ and $\sigma_{k-1}(\kappa|j)$, we have
following estimate
\begin{align}\label{n3.12}
\sigma_{k-1}(\kappa|i)=&\sigma_{k-1}(\kappa|j)-t\sigma_{k-2}(\kappa|ij)\\
=&\sigma_{k-1}(\kappa|j)-\dfrac{t}{\kappa_i}[\sigma_{k-1}(\kappa|j)-\sigma_{k-1}(\kappa|ij)]\nonumber\\
=&\left(1-\dfrac{t}{\kappa_i}\right)\sigma_{k-1}(\kappa|j)+\dfrac{t}{\kappa_i}\sigma_{k-1}(\kappa|ij)\nonumber\\
\geq&\left(1-\dfrac{t(1+\Theta)}{\kappa_i}\right)\sigma_{k-1}(\kappa|j)\geq\dfrac{1}{2}\sigma_{k-1}(\kappa|j),\nonumber
\end{align}
if $\kappa_1$ is sufficiently large. Here in the fourth inequality,
we have used Lemma \ref{lemm7} to estimate the term
$\sigma_{k-1}(\kappa|ij)$. We also have
\begin{align}\label{n3.13}
\kappa_j^2\tilde\sigma_{k-2}+\sigma_k(\kappa)-\tilde\sigma_{k}
=&(\kappa_j+\kappa_i)\sigma_{k-1}(\kappa|i)
\end{align}
and
\begin{align}\label{n315}
3\kappa_i=3\kappa_j+3t\leq 3\kappa_j+3\sqrt{\kappa_1}\leq 4\kappa_j,
\end{align}
where we have used $\kappa_j>\kappa_1/2$ and $\kappa_1$ is
sufficiently large. Thus, using \eqref{n3.13} and \eqref{n315},
(\ref{s3.08}) becomes
\begin{align*}
L \geq&(e^t-1)(\kappa_j+\kappa_i)\sigma_{k-1}(\kappa|i)
-3t\kappa_i\tilde\sigma_{k-2}\nonumber\\
\geq&t(\kappa_j+\kappa_i)\sigma_{k-1}(\kappa|i)
-4t\kappa_j\tilde\sigma_{k-2}\nonumber\\
=&t(\kappa_j+\kappa_i-4)\sigma_{k-1}(\kappa|i)+4t\sigma_{k-1}(\kappa|ij)\\
\geq&t(\kappa_j+\kappa_i-4)\sigma_{k-1}(\kappa|j)/2+4t\sigma_{k-1}(\kappa|ij)\\
\geq&\frac{t}{2}(\kappa_j+\kappa_i-4-8\Theta)\sigma_{k-1}(\kappa|j)\geq
0
\end{align*}
if $\kappa_1$ is sufficiently large. Here in the forth inequality,
we have used \eqref{n3.12} and in the last inequality, we have used
Lemma \ref{lemm7} to give the lower bound of
$\sigma_{k-1}(\kappa|ij)$.


If $\kappa_i<\kappa_j$, we let $t=\kappa_j-\kappa_i$, which yields
$0<t\leq \kappa_1-\kappa_i<\sqrt{\kappa_1}/n$. For simplification
purpose, we still denote $\tilde\sigma_m=\sigma_m(\kappa|ij)$. Thus,
using $t=\kappa_j-\kappa_i$ we have
\begin{align}\label{s3.09}
L=&2\kappa_ie^{t}\sigma_{k-1}(\kappa|j)-(\kappa_i+\kappa_j)\sigma_{k-1}(\kappa|i)\\
=&2\kappa_ie^{t}(\kappa_i\tilde\sigma_{k-2}+\tilde\sigma_{k-1})-(\kappa_i+\kappa_j)(\kappa_j\tilde\sigma_{k-2}+\tilde\sigma_{k-1})\nonumber\\
=&[2\kappa_i^2(e^{t}-1)-3t\kappa_i-t^2]\tilde\sigma_{k-2}+[2\kappa_i(e^t-1)-t]\tilde\sigma_{k-1}.\nonumber
\end{align}
We divide into two cases to prove $L\geq 0$.

Case (c1): Suppose $\tilde\sigma_{k-1}\geq 0$. Since we have
$e^t>1+t$, $t<\sqrt{\kappa_1}/n$ and $\kappa_j>\kappa_i>\kappa_1/2$,
in view of (\ref{s3.09}), we get $L\geq 0$.

Case (c2): Suppose $\tilde\sigma_{k-1}< 0$. Inserting the identity
\eqref{n311} into the last formula of (\ref{s3.09}) and using
$2\kappa_i=\kappa_i+\kappa_j-t$, we get
\begin{align}\label{n3.16}
L=&[2\kappa_i^2(e^{t}-1)-3t\kappa_i-t^2]\tilde\sigma_{k-2}-te^t\tilde\sigma_{k-1}\\
&+(e^t-1)[\sigma_k(\kappa)-\kappa_i\kappa_j\tilde\sigma_{k-2}-\tilde\sigma_{k}]\nonumber\\
\geq&[(\kappa_i^2-\kappa_i
t)(e^t-1)-4t\kappa_i]\tilde\sigma_{k-2}+(e^t-1)[\sigma_k(\kappa)-\tilde\sigma_{k}],\nonumber
\end{align}
where in the last inequality, we have used $\tilde\sigma_{k-1}<0$,
$t<\sqrt{\kappa_1}/n<\kappa_i$ and $\kappa_1$ is sufficiently large.
Note that comparing the previous case $\kappa_i>\kappa_j$, this case
exchanges the position of $i$ and $j$. Thus, exchanging the indices
$i$ and $j$ in \eqref{n3.13} and \eqref{n3.12}, we get the formulae,
\begin{align}\label{n3.17}
\kappa_i^2\tilde\sigma_{k-2}+\sigma_k(\kappa)-\tilde\sigma_{k}=&(\kappa_i+\kappa_j)\sigma_{k-1}(\kappa|j),
\ \ \sigma_{k-1}(\kappa|j)\geq \frac{\sigma_{k-1}(\kappa|i)}{2}.
\end{align}
Combing \eqref{n3.16} with \eqref{n3.17}, we get
\begin{align*}
L \geq&(e^t-1)(\kappa_i+\kappa_j)\sigma_{k-1}(\kappa|j)
-[(e^t-1)t+4t]\kappa_i\tilde\sigma_{k-2}\nonumber\\
=&[(e^t-1)(\kappa_i+\kappa_j)-(e^t+3)t]\sigma_{k-1}(\kappa|j)
+(e^t+3)t\tilde\sigma_{k-1}\\
\geq&[(e^t-1)(\kappa_i+\kappa_j)-(e^t+3)t]\frac{\sigma_{k-1}(\kappa|i)}{2}
+(e^t+3)t\tilde\sigma_{k-1}\\
\geq&[(e^t-1)(\kappa_i+\kappa_j)-(1+2\Theta)(e^t+3)t]\frac{\sigma_{k-1}(\kappa|i)}{2}\\
\geq&0,
\end{align*}
if $\kappa_1$ is sufficiently large. Here, in the fourth inequality,
we  have used Lemma \ref{lemm7} to give the lower bound of
$\tilde\sigma_{k-1}$,  and in the last inequality, we have used
$e^t>1+t$, $\kappa_j\geq\kappa_i\geq\kappa_1-\sqrt{\kappa_1}/n$,
$t\leq \sqrt{\kappa_1}/n$.
\end{proof}

Now, we are in the position to handle the case (II).

\begin{lemm}\label{le1}
Assume $\kappa\in\Gamma_k$, $2k>n$, and $\kappa_1$ is the maximum
entry of $\kappa$. For given index $1\leq i\leq n$, if $\kappa_i>
\kappa_1-\sqrt{\kappa_1}/n$, then we have
\begin{align*}
A_i+B_i+C_i+D_i-E_i\geq 0,
\end{align*}
when the positive constant $K$ and the biggest eigenavlue $\kappa_1$ both are sufficiently large.
\end{lemm}
\begin{proof}
Using  (\ref{e2.20}), we have
\begin{align}\label{e2.27}
&A_i+B_i+C_i+D_i-E_i\\
\geq& e^{\kappa_i}\big[K(\sigma_{k})_i^2-\sigma_{k}^{pp,qq}h_{ppi}h_{qqi}\big]+2\dsum_{l\neq i}e^{\kappa_l}\sigma_{k}^{ll,ii}h_{lli}^2\nonumber\\
&-\dfrac{1}{\log P}\dsum_{l\neq
i}e^{\kappa_l}\sigma_{k}^{ii}h_{lli}^2+\dfrac{1+\log P}{P\log
P}\dsum_{l\neq
i}e^{\kappa_l+\kappa_i}\sigma_{k}^{ii}h_{lli}^2 \nonumber\\
&+ 2\dsum_{l\neq
i}\dfrac{e^{\kappa_l}-e^{\kappa_i}}{\kappa_l-\kappa_i}\sigma_{k}^{ll}h_{lli}^2
+e^{\kappa_i}\sigma_{k}^{ii}h_{iii}^2 -\frac{1+\log
P}{P\log P}e^{2\kappa_i}\sigma_{k}^{ii}h_{iii}^2 \nonumber\\
&-2\frac{1+\log P}{P\log P}\dsum_{l\neq
i}e^{\kappa_i+\kappa_l}\sigma_{k}^{ii}h_{iii}h_{lli}. \nonumber
\end{align}

Note that $\log P>\kappa_1$, using Conjecture \ref{con} and Lemma
\ref{lemm16}, we have
\begin{align}\label{e2.28}
e^{\kappa_i}\big[K(\sigma_{k})_i^2-\sigma_{k}^{pp,qq}h_{ppi}h_{qqi}\big]+2\dsum_{l\neq
i}\dfrac{e^{\kappa_l}-e^{\kappa_i}}{\kappa_l-\kappa_i}\sigma_{k}^{ll}h_{lli}^2
\geq\frac{1}{\log P}e^{\kappa_i}\sigma_{k}^{ii}h_{iii}^2.
\end{align}

Now, we will show
\begin{align}\label{e2.30}
2\dsum_{l\neq i}e^{\kappa_l}\sigma_{k}^{ll,ii}h_{lli}^2
-\dfrac{1}{\log P}\dsum_{l\neq
i}e^{\kappa_l}\sigma_{k}^{ii}h_{lli}^2\geq 0.
\end{align}
It is clear that
\begin{align}\label{s3.21}
2\kappa_1\sigma_{k}^{ii,ll}(\kappa)-\sigma_{k}^{ii}(\kappa)\geq
&2\kappa_1\sigma_{k-2}(\kappa|1i)-\sigma_{k-1}(\kappa|i)\\
=&\kappa_1\sigma_{k-2}(\kappa|1i)-\sigma_{k-1}(\kappa|1i).\nonumber
\end{align}
Thus, if $\sigma_{k-1}(\kappa|1i)\leq 0$, (\ref{e2.30}) obviously
holds. If $\sigma_{k-1}(\kappa|1i)>0$, we have
$(\kappa|1i)\in\Gamma_{k-1}$. Therefore, by Lemma \ref{lemm8}, we
have
\begin{align}\label{s3.22}
\sigma_{k-1}(\kappa|1i)\leq
\dfrac{n-k}{k-1}\sigma_{k-2}(\kappa|1i)\kappa_1.
\end{align}
Thus if we have $n< 2k$ which implies $n\leq 2k-1$, combing
(\ref{s3.21}) with (\ref{s3.22}), we get
\begin{center}
$ 2\kappa_1\sigma_{k}^{ii,ll}(\kappa)-\sigma_{k}^{ii}(\kappa)\geq
\dfrac{2k-1-n}{k-1}\sigma_{k-2}(\kappa|1i)\kappa_1\geq 0, $
\end{center} which gives the inequality (\ref{e2.30}).

On the other hand, we have
\begin{align}\label{e2.31}
&\dfrac{1+\log P}{P\log P}\dsum_{l\neq
i}e^{\kappa_l+\kappa_i}\sigma_{k}^{ii}h_{lli}^2 -2\frac{1+\log
P}{P\log P}\dsum_{l\neq
i}e^{\kappa_i+\kappa_l}\sigma_{k}^{ii}h_{iii}h_{lli}\\
\geq&-\dfrac{1+\log P}{P\log P}\dsum_{l\neq
i}e^{\kappa_l+\kappa_i}\sigma_{k}^{ii}h_{iii}^2. \nonumber
\end{align}
Inserting (\ref{e2.28}), (\ref{e2.30}) and (\ref{e2.31}) into
(\ref{e2.27}), we obtain
\begin{align}
&A_i+B_i+C_i+D_i-E_i\nonumber\\
\geq& \frac{1}{\log
P}e^{\kappa_i}\sigma_{k}^{ii}h_{iii}^2+e^{\kappa_i}\sigma_{k}^{ii}h_{iii}^2
-\frac{1+\log P}{P\log
P}e^{2\kappa_i}\sigma_{k}^{ii}h_{iii}^2\nonumber\\
&-\dfrac{1+\log P}{P\log P}\dsum_{l\neq
i}e^{\kappa_l+\kappa_i}\sigma_{k}^{ii}h_{iii}^2\nonumber\\
=&0\nonumber.
\end{align}

\end{proof}

\section{An inequality}
In this section, we will prove Theorem \ref{lemm17}.
The argument is closed to \cite{RW}, but will become
a little more complicated.

Before to prove our Theorem, we need some algebraic identities.
\begin{lemm}\label{lemm18}
Assume $\kappa=(\kappa_1,\cdots,\kappa_n)\in\Gamma_{k}$. Suppose $1\leq i,j,p,q\leq n$ are given indices.
$a_j$ and $c_{k,K}$ are defined by \eqref{aj} and \eqref{ckK}. We have the following five identities:

\noindent (1)
\begin{align*}
&\kappa_iK\sigma_{k}^{ii}(\kappa)\sigma_{k}^{jj}(\kappa)[-\sigma_{k}^{jj}(\kappa)+2\kappa_i\sigma_{k}^{ii,jj}(\kappa)]-\kappa_i^2[\sigma_{k}^{ii,jj}(\kappa)]^2
+a_j[\kappa_iK\sigma_{k}^{ii}(\kappa)-1]\sigma_{k}^{ii}(\kappa)\\
=&\frac{1}{c_{k,K}}[\sigma_{k}^{ii}(\kappa)+\sigma_{k}^{jj}(\kappa)](\kappa_i+\kappa_j)\sigma_{k-2}(\kappa|ij)-\sigma_{k-1}^2(\kappa|ij).
\end{align*}

\noindent (2)
\begin{align*}
&\kappa_i\left[\sigma_{k}^{pp}(\kappa)\sigma_{k}^{ii,qq}(\kappa)+\sigma_{k}^{qq}(\kappa)\sigma_{k}^{ii,pp}(\kappa)-\sigma_{k}^{ii}(\kappa)
\sigma_{k}^{pp,qq}(\kappa)\right]-\sigma_{k}^{pp}(\kappa)\sigma_{k}^{qq}(\kappa)\nonumber\\
&-\kappa_i^2\sigma_{k}^{ii,pp}(\kappa)\sigma_{k}^{ii,qq}(\kappa)+\kappa_i\sigma_{k}^{ii}(\kappa)\sigma_{k}^{pp,qq}(\kappa)\nonumber\\
=&-\sigma_{k-1}(\kappa|ip)\sigma_{k-1}(\kappa|iq).
\end{align*}

\noindent (3)
\begin{align*}
\left[\sigma_{k}^{ii}(\kappa)+\sigma_{k}^{jj}(\kappa)\right](\kappa_i+\kappa_j)
=&2\sigma_k(\kappa)-2\sigma_{k}(\kappa|ij)+\kappa_i^2\sigma_{k-2}(\kappa|ij)+\kappa_j^2\sigma_{k-2}(\kappa|ij).\nonumber
\end{align*}

\noindent (4)
\begin{align*}
\sigma_{k}^{qq}(\kappa)\sigma_{k}^{ii,pp}(\kappa)-\sigma_{k}^{ii}(\kappa)\sigma_{k}^{pp,qq}(\kappa)=&\kappa_i\sigma_{k-2}^2(\kappa|ipq)
-\kappa_i\sigma_{k-1}(\kappa|ipq)\sigma_{k-3}(\kappa|ipq)\nonumber\\
&+\kappa_q\sigma_{k-3}(\kappa|ipq)\sigma_{k-1}(\kappa|ipq)
-\kappa_q\sigma_{k-2}^2(\kappa|ipq).
\end{align*}

\noindent (5)
\begin{align*}
\sigma_{k}^{pp}(\kappa)\sigma_{k-1}(\kappa|iq)=&\sigma_k\sigma_{k-2}(\kappa|ipq)+\sigma_{k-1}^2(\kappa|ipq)-\sigma_{k}(\kappa|ipq)\sigma_{k-2}(\kappa|ipq)\nonumber\\
&-\kappa_q\kappa_i\sigma_{k-2}^2(\kappa|ipq)+\kappa_q\kappa_i\sigma_{k-3}(\kappa|ipq)\sigma_{k-1}(\kappa|ipq).
\end{align*}

\end{lemm}
\par
\begin{proof}For simplification purpose, we omit the $\kappa$ in our notations in the following argument, which means that we let
$$\sigma_k=\sigma_k(\kappa), \ \ \sigma_k^{pp}=\sigma_k^{pp}(\kappa), \ \ \sigma_k^{pp,qq}=\sigma_k^{pp,qq}(\kappa).$$

\noindent (1) Using the identity
\begin{align*}
-\sigma_{k}^{jj}+2\kappa_i\sigma_{k}^{ii,jj}+\sigma_{k}^{ii}=(\kappa_i+\kappa_j)\sigma_{k-2}(\kappa|ij),
\end{align*}
and $a_j=\sigma_{k}^{jj}+(\kappa_i+\kappa_j)\sigma_k^{ii,jj}$, we
have
\begin{align}
&\kappa_iK\sigma_{k}^{ii}\sigma_{k}^{jj}(-\sigma_{k}^{jj}+2\kappa_i\sigma_{k}^{ii,jj})-\kappa_i^2(\sigma_{k}^{ii,jj})^2
+a_j[\kappa_iK(\sigma_{k}^{ii})^2-\sigma_{k}^{ii}] \label{s4.02}\\
=&\kappa_iK\sigma_{k}^{ii}\sigma_{k}^{jj}(-\sigma_{k}^{jj}+2\kappa_i\sigma_{k}^{ii,jj}+\sigma_{k}^{ii})-\kappa_i^2(\sigma_{k}^{ii,jj})^2-\sigma_{k}^{ii}\sigma_{k}^{jj}\nonumber\\
&
+(\kappa_iK\sigma_{k}^{ii}-1)\sigma_{k}^{ii}(\kappa_i+\kappa_j)\sigma_k^{ii,jj} \nonumber\\
=&(\kappa_iK\sigma_{k}^{ii}-1)\sigma_{k}^{jj}(\kappa_i+\kappa_j)\sigma_{k-2}(\kappa|ij)+(\kappa_i+\kappa_j)\sigma_k^{jj}\sigma_{k-2}(\kappa|ij)\nonumber
\\ &
+(\kappa_iK\sigma_{k}^{ii}-1)\sigma_{k}^{ii}(\kappa_i+\kappa_j)\sigma_k^{ii,jj}-\kappa_i^2(\sigma_{k}^{ii,jj})^2-\sigma_{k}^{ii}\sigma_{k}^{jj} \nonumber \\
=&(\kappa_iK\sigma_{k}^{ii}-1)(\sigma_{k}^{ii}+\sigma_{k}^{jj})(\kappa_i+\kappa_j)\sigma_{k-2}(\kappa|ij)\nonumber\\
&+(\kappa_i+\kappa_j)\sigma_{k}^{jj}\sigma_{k-2}(\kappa|ij)-\kappa_i^2(\sigma_{k}^{ii,jj})^2-\sigma_{k}^{ii}\sigma_{k}^{jj}
\nonumber\\
=&(\kappa_iK\sigma_{k}^{ii}-1)(\sigma_{k}^{ii}+\sigma_{k}^{jj})(\kappa_i+\kappa_j)\sigma_{k-2}(\kappa|ij)-\sigma_{k-1}^2(\kappa|ij).\nonumber
\end{align}
Here in the above last equality we have used the following identity,
\begin{align*}
&(\kappa_i+\kappa_j)\sigma_{k}^{jj}\sigma_{k-2}(\kappa|ij)-\kappa_i^2(\sigma_{k}^{ii,jj})^2-\sigma_{k}^{ii}\sigma_{k}^{jj}\\
=&\kappa_i\sigma_{k-2}(\kappa|ij)[\sigma_{k-1}(\kappa|j)-\kappa_i\sigma_{k-2}(\kappa|ij)]+\sigma_{k}^{jj}[\kappa_j\sigma_{k-2}(\kappa|ij)-\sigma_{k-1}(\kappa|i)]\\
=&\kappa_i\sigma_{k-2}(\kappa|ij)\sigma_{k-1}(\kappa|ij)-\sigma_{k}^{jj}\sigma_{k-1}(\kappa|ij)\\
=&-\sigma_{k-1}^2(\kappa|ij).
\end{align*}
\par

\noindent (2) We have
\begin{align}
&\kappa_i(\sigma_{k}^{pp}\sigma_{k}^{ii,qq}+\sigma_{k}^{qq}\sigma_{k}^{ii,pp}-\sigma_{k}^{ii}\sigma_{k}^{pp,qq})-\sigma_{k}^{pp}\sigma_{k}^{qq}
-\kappa_i^2\sigma_{k}^{ii,pp}\sigma_{k}^{ii,qq}+\kappa_i\sigma_{k}^{ii}\sigma_{k}^{pp,qq}\nonumber\\
=&\kappa_i\sigma_{k}^{qq}\sigma_{k}^{ii,pp}-\sigma_{k}^{pp}\sigma_{k-1}(\kappa|iq)
-\kappa_i^2\sigma_{k}^{ii,pp}\sigma_{k}^{ii,qq}\nonumber\\
=&\kappa_i\sigma_{k-1}(\kappa|iq)\sigma_k^{ii,pp}-\sigma_{k}^{pp}\sigma_{k-1}(\kappa|iq)\nonumber\\
=&-\sigma_{k-1}(\kappa|ip)\sigma_{k-1}(\kappa|iq). \nonumber
\end{align}

\noindent (3) We have
\begin{align*}
(\sigma_{k}^{ii}+\sigma_{k}^{jj})(\kappa_i+\kappa_j)
=&\kappa_i\sigma_{k-1}(\kappa|i)+\kappa_j\sigma_{k-1}(\kappa|j)+\kappa_i\sigma_{k-1}(\kappa|j)+\kappa_j\sigma_{k-1}(\kappa|i)\\
=&2\sigma_k-\sigma_{k}(\kappa|i)-\sigma_{k}(\kappa|j)+\kappa_i^2\sigma_{k-2}(\kappa|ij)+\kappa_i\sigma_{k-1}(\kappa|ij)\\
&+\kappa_j^2\sigma_{k-2}(\kappa|ij)+\kappa_j\sigma_{k-1}(\kappa|ij)\\
=&2\sigma_k-2\sigma_{k}(\kappa|ij)+\kappa_i^2\sigma_{k-2}(\kappa|ij)+\kappa_j^2\sigma_{k-2}(\kappa|ij).
\end{align*}

\noindent (4) We further denote $\tilde\sigma_{m}=\sigma_{m}(\kappa|ipq)$ here. Thus, we have
\begin{align*}
&\sigma_{k}^{qq}\sigma_{k}^{ii,pp}-\sigma_{k}^{ii}\sigma_{k}^{pp,qq}\\
=&[\kappa_p\sigma_{k-2}(\kappa|pq)+\sigma_{k-1}(\kappa|pq)]\sigma_{k-2}(\kappa|ip)-\sigma_{k-2}(\kappa|pq)[\kappa_p\sigma_{k-2}(\kappa|ip)+\sigma_{k-1}(\kappa|ip)]\\
=&\sigma_{k-1}(\kappa|pq)\sigma_{k-2}(\kappa|ip)-\sigma_{k-2}(\kappa|pq)\sigma_{k-1}(\kappa|ip)\\
=&(\kappa_i\tilde\sigma_{k-2}+\tilde\sigma_{k-1})(\kappa_q\tilde\sigma_{k-3}+\tilde\sigma_{k-2})-(\kappa_i\tilde\sigma_{k-3}+\tilde\sigma_{k-2})(\kappa_q\tilde\sigma_{k-2}+\tilde\sigma_{k-1})\\
=&\kappa_i\tilde\sigma_{k-2}^2-\kappa_i\tilde\sigma_{k-1}\tilde\sigma_{k-3}+\kappa_q\tilde\sigma_{k-3}\tilde\sigma_{k-1}-\kappa_q\tilde\sigma_{k-2}^2.
\end{align*}

\noindent (5) We also denote $\tilde\sigma_{m}=\sigma_{m}(\kappa|ipq)$ here. We have
\begin{align*}
\sigma_{k}^{pp}\sigma_{k-1}(\kappa|iq)
=&\sigma_{k-1}(\kappa|p)[\kappa_p\sigma_{k-2}(\kappa|ipq)+\sigma_{k-1}(\kappa|ipq)]\\
=&\kappa_p\sigma_{k-1}(\kappa|p)\tilde\sigma_{k-2}+\sigma_{k-1}(\kappa|p)\tilde\sigma_{k-1}\\
=&[\sigma_k-\sigma_{k}(\kappa|p)]\tilde\sigma_{k-2}+[\kappa_q\kappa_i\tilde\sigma_{k-3}+(\kappa_q+\kappa_i)\tilde\sigma_{k-2}+\tilde\sigma_{k-1}]\tilde\sigma_{k-1}\\
=&\sigma_k\tilde\sigma_{k-2}-[\kappa_q\kappa_i\tilde\sigma_{k-2}+(\kappa_q+\kappa_i)\tilde\sigma_{k-1}+\tilde\sigma_{k}]
\tilde\sigma_{k-2}\\
&+[\kappa_q\kappa_i\tilde\sigma_{k-3}+(\kappa_q+\kappa_i)\tilde\sigma_{k-2}+\tilde\sigma_{k-1}]\tilde\sigma_{k-1}\\
=&\sigma_k\tilde\sigma_{k-2}+\tilde\sigma_{k-1}^2-\tilde\sigma_{k}\tilde\sigma_{k-2}-\kappa_q\kappa_i\tilde\sigma_{k-2}^2
+\kappa_q\kappa_i\tilde\sigma_{k-3}\tilde\sigma_{k-1}.
\end{align*}

\end{proof}

\noindent{\bf Proof of the Theorem \ref{lemm17}}:
For the sake of simplification, we still omit the $\kappa$ in the following calculation, which means that we still let
$$\sigma_k=\sigma_k(\kappa),\ \  \sigma_k^{pp}=\sigma_k^{pp}(\kappa), \ \ \sigma_k^{pp,qq}=\sigma_k^{pp,qq}(\kappa).$$
Let's calculate the left hand side of \eqref{s3.01}. By Lemma \ref{lemm11}, we have
$$K\kappa_i\sigma_{k}^{ii}-1\geq K\kappa_1\sigma_k^{11}/2-1\geq K\theta\sigma_{k}/2-1\geq 0,$$  if the positive constant $K$ and $\kappa_1$ both are sufficiently large.
A
straightforward calculation shows,
\begin{align}\label{s4.03}
&\kappa_i\Big[K\Big(\sum_{j}\sigma_{k}^{jj}\xi_j\Big)^2-\sigma_{k}^{pp,qq}\xi_p\xi_q\Big]-\sigma_{k}^{ii}\xi_i^2+\sum_{j\neq i}a_j\xi_j^2\\
=&\kappa_iK\Big(\sum_{j\neq i}\sigma_{k}^{jj}\xi_j\Big)^2+2\kappa_i\xi_i\Big[\sum_{j\neq i}(K\sigma_{k}^{ii}\sigma_{k}^{jj}-\sigma_{k}^{ii,jj})\xi_j\Big]\nonumber\\
&+[\kappa_iK(\sigma_{k}^{ii})^2-\sigma_{k}^{ii}]\xi_i^2+\sum_{j\neq i}a_j\xi_j^2-\kappa_i\sum_{p\neq i; q\neq i}\sigma_{k}^{pp,qq}\xi_p\xi_q\nonumber\\
\geq&\kappa_iK\Big(\sum_{j\neq i}\sigma_{k}^{jj}\xi_j\Big)^2-\frac{\kappa_i^2\Big[\sum_{j\neq i}(K\sigma_{k}^{ii}\sigma_{k}^{jj}-\sigma_{k}^{ii,jj})\xi_j\Big]^2}{\kappa_iK(\sigma_{k}^{ii})^2-\sigma_{k}^{ii}}\nonumber\\
&+\sum_{j\neq i}a_j\xi_j^2-\kappa_i\sum_{p\neq i; q\neq
i}\sigma_{k}^{pp,qq}\xi_p\xi_q\nonumber\\
=&\sum_{j\neq
i}\left[\kappa_iK(\sigma_{k}^{jj})^2-\frac{\kappa_i^2(K\sigma_{k}^{ii}\sigma_{k}^{jj}-\sigma_{k}^{ii,jj})^2}{\kappa_iK(\sigma_{k}^{ii})^2-\sigma_{k}^{ii}}+a_j\right]\xi_j^2\nonumber\\
&+\sum_{p,q\neq i;p\neq
q}\Big[\kappa_iK\sigma_{k}^{pp}\sigma_{k}^{qq}-\frac{\kappa_i^2(K\sigma_{k}^{ii}\sigma_{k}^{pp}
-\sigma_{k}^{ii,pp})(K\sigma_{k}^{ii}\sigma_{k}^{qq}-\sigma_{k}^{ii,qq})}{\kappa_iK(\sigma_{k}^{ii})^2
-\sigma_{k}^{ii}}-\kappa_i\sigma_{k}^{pp,qq}\Big]\xi_p\xi_q\nonumber,
\end{align}
where, in the second inequality, we have used,
\begin{align*}
&\frac{\kappa_i^2\Big[\sum_{j\neq
i}(K\sigma_{k}^{ii}\sigma_{k}^{jj}-\sigma_{k}^{ii,jj})\xi_j\Big]^2}{\kappa_iK(\sigma_{k}^{ii})^2-\sigma_{k}^{ii}}+2\kappa_i\xi_i\Big[\sum_{j\neq
i}(K\sigma_{k}^{ii}\sigma_{k}^{jj}-\sigma_{k}^{ii,jj})\xi_j\Big]\nonumber\\
&+[\kappa_iK(\sigma_{k}^{ii})^2-\sigma_{k}^{ii}]\xi_i^2\geq 0.
\nonumber
\end{align*}

Thus, we can multiple the term
$\kappa_iK(\sigma_{k}^{ii})^2-\sigma_{k}^{ii}$ in both sides of
\eqref{s4.03}. Then, we get
\begin{align}\label{s4.04}
&[\kappa_iK(\sigma_{k}^{ii})^2-\sigma_{k}^{ii}]\Big\{\kappa_i\Big[K\Big(\sum_{j}\sigma_{k}^{jj}\xi_j\Big)^2-\sigma_{k}^{pp,qq}\xi_p\xi_q\Big]-\sigma_{k}^{ii}\xi_i^2
+\sum_{j\neq i}a_j\xi_j^2\Big\}\\
\geq&\sum_{j\neq
i}\Big[\kappa_iK\sigma_{k}^{ii}\sigma_{k}^{jj}(-\sigma_{k}^{jj}+2\kappa_i\sigma_{k}^{ii,jj})-\kappa_i^2(\sigma_{k}^{ii,jj})^2
+a_j(\kappa_iK\sigma_{k}^{ii}-1)\sigma_{k}^{ii}\Big]\xi_j^2\nonumber\\
&+\sum_{p,q\neq i;p\neq
q}\Big[\kappa_iK\sigma_{k}^{ii}[\kappa_i(\sigma_{k}^{pp}\sigma_{k}^{ii,qq}+\sigma_{k}^{qq}\sigma_{k}^{ii,pp}-\sigma_{k}^{ii}\sigma_{k}^{pp,qq})-\sigma_{k}^{pp}\sigma_{k}^{qq}]\nonumber\\
&-\kappa_i^2\sigma_{k}^{ii,pp}\sigma_{k}^{ii,qq}+\kappa_i\sigma_{k}^{ii}\sigma_{k}^{pp,qq}\Big]\xi_p\xi_q\nonumber
\\
=&\dsum_{j\neq
i}\Big[(\kappa_iK\sigma_{k}^{ii}-1)(\sigma_{k}^{ii}+\sigma_{k}^{jj})(\kappa_i+\kappa_j)\sigma_{k-2}(\kappa|ij)-\sigma_{k-1}^2(\kappa|ij)\Big]\xi_j^2\nonumber\\
&+\sum_{p,q\neq i;p\neq
q}\Big[(\kappa_iK\sigma_{k}^{ii}-1)[\kappa_i(\sigma_{k}^{pp}\sigma_{k}^{ii,qq}+\sigma_{k}^{qq}\sigma_{k}^{ii,pp}-\sigma_{k}^{ii}\sigma_{k}^{pp,qq})-\sigma_{k}^{pp}\sigma_{k}^{qq}]\nonumber\\
&-\sigma_{k-1}(\kappa|ip)\sigma_{k-1}(\kappa|iq)\Big]\xi_p\xi_q\nonumber\\
= &\dsum_{j\neq
i}\Big[(\kappa_iK\sigma_{k}^{ii}-1)(\sigma_{k}^{ii}+\sigma_{k}^{jj})(\kappa_i+\kappa_j)\sigma_{k-2}(\kappa|ij)-\sigma_{k-1}^2(\kappa|ij)\Big]\xi_j^2\nonumber\\
&+\sum_{p,q\neq i;p\neq
q}\Big[(\kappa_iK\sigma_{k}^{ii}-1)(\kappa_i\sigma_{k}^{qq}\sigma_{k}^{ii,pp}-\kappa_i\sigma_{k}^{ii}\sigma_{k}^{pp,qq}-\sigma_{k}^{pp}\sigma_{k-1}(\kappa|iq))\nonumber\\
&-\sigma_{k-1}(\kappa|ip)\sigma_{k-1}(\kappa|iq)\Big]\xi_p\xi_q\nonumber\\
=&(\kappa_iK\sigma_{k}^{ii}-1)\dsum_{j\neq
i}\Big[\kappa_i^2\sigma_{k-2}^2(\kappa|ij)+\kappa_j^2\sigma_{k-2}^2(\kappa|ij)-2\sigma_{k}(\kappa|ij)\sigma_{k-2}(\kappa|ij)\nonumber\\
&+2\sigma_{k}\sigma_{k-2}(\kappa|ij)
-c_{k,K}\sigma_{k-1}^2(\kappa|ij)\Big]\xi_j^2\nonumber\\
&+(\kappa_iK\sigma_{k}^{ii}-1)\sum_{p,q\neq i;p\neq
q}\Big[\kappa_i^2\sigma_{k-2}^2(\kappa|ipq)
-\kappa_i^2\sigma_{k-1}(\kappa|ipq)\sigma_{k-3}(\kappa|ipq)\nonumber\\
&-\sigma_k\sigma_{k-2}(\kappa|ipq) +\sigma_{k}(\kappa|ipq)\sigma_{k-2}(\kappa|ipq)-\sigma_{k-1}^2(\kappa|ipq)\nonumber\\
&-c_{k,K}\sigma_{k-1}(\kappa|ip)\sigma_{k-1}(\kappa|iq)\Big]\xi_p\xi_q\nonumber\\
=&\frac{1}{c_{k,K}}\left[\kappa_i^2\textbf{A}_{k;i}+\sigma_{k}\textbf{B}_{k;i}+\textbf{C}_{k;i}-c_{k,K}\textbf{D}_{k;i}\right],\nonumber
\end{align}
where in the second equality, we have used identities (1),(2) of Lemma \ref{lemm18} and in the forth equality, we have used identities (3),(4),(5) of  Lemma \ref{lemm18}.  We have completed our proof.
\bigskip

\noindent{\it Notes:} The present paper is a plenary version of the section 3 and section 4 in \cite{RWarxiv}. We also would like to mention an interesting paper by Chu \cite{Chu} which appeared after \cite{RWarxiv}.

\end{document}